\documentclass{amsart}

\newtheorem{theorem}{Theorem}[section]
\newtheorem{lemma}[theorem]{Lemma}
\newtheorem{prop}[theorem]{Proposition}
\newtheorem{cor}[theorem]{Corollary}

\theoremstyle{definition}
\newtheorem{definition}[theorem]{Definition}
\newtheorem{example}[theorem]{Example}

\theoremstyle{remark}

\numberwithin{equation}{section}

\begin{document}

\title{Basis entropy in Banach spaces}

\author{Andrei Dorogovtsev}
\address{Institute of Mathematics, National Academy of Sciences, Kyiv, Ukraine}
\email{adoro@imath.kiev.ua}

\author{Mikhail Popov}
\address{Department of Mathematics and Informatics, Chernivtsi National University,
Chernivtsi, Ukraine}
\email{misham.popov@gmail.com}


\date{September 1, 2013 and, in revised form, ...}



\begin{abstract}
To be decided
\end{abstract}

\maketitle

\section{Introduction}

All linear (e.g., normed, Banach, Hilbert) spaces are considered over the reals.

A subset $A$ of a Banach space $X$ is called \emph{precompact} if for every $\varepsilon > 0$, $A$ contains a finite $\varepsilon$-net, that is, a finite collection $x_1, \ldots, x_m \in A$ with $(\forall x \in A)(\exists k \in \{1, \ldots, m\}) \|x - x_k\| \leq \varepsilon$. By the well known Hausdorff criterion (which is valid for metric spaces), $A$ is precompact if and only if its norm closure $\overline{A}$ is compact in $X$.

\subsection{Necessary information on bases in Banach spaces}

\subsubsection{Bases in Banach spaces}

We follow mainly \cite{LTI} (see also \cite{AK}, \cite{S}). Recall that a sequence $(e_n)_{n=1}^\infty$ in $X$ is called a \emph{basis} (more precisely, a \emph{Schauder basis}) of $X$ if for every $x \in X$ there is a unique sequence of scalars $(a_n)_{n=1}^\infty$ such that $x = \sum_{n=1}^\infty a_n e_n$. In this case, the coefficients $a_n = e_n^*(x)$ are continuous linear functionals of $x$ and called \emph{biorthogonal functionals}. So, $x = \sum_{n=1}^\infty e_n^*(x) \, e_n$ for each $x \in X$. The biorthogonal functionals possess the following property: $e_i^*(e_j) = \delta_{i,j}$. Moreover, this property determines the biorthogonal functionals: for every sequence $(f_n)_{n=1}^\infty$ in $X^*$ the condition $f_i(e_j) = \delta_{i,j}$ for all $i,j$ implies that $f_i = e_i^*$ for all $i$. The partial sums projections $P_n$ of $X$ defined by $P_n x = \sum_{k=1}^n e_k^*(x) \, e_k$, $x \in X$, calling the \emph{basis projections}, are uniformly bounded in $n$, and the number $K = \sup_n \|P_n\| < \infty$ is called the \emph{basis constant} of $(e_n)_{n=1}^\infty$. A \emph{basic sequence} is any sequence $(e_n)_{n=1}^\infty$ in $X$ which is a basis of some subspace $X_0$ of $X$ (more precisely, a basis of its closed linear span $[e_n]_{n=1}^\infty$). A sequence $(e_n)_{n=1}^\infty$ of nonzero elements of $X$ is a basic sequence if and only if there is a number $K \in [1,+\infty)$ such that
$$
\Bigl\| \sum\limits_{k=1}^n a_k e_k \Bigr\| \leq K \, \Bigl\| \sum\limits_{k=1}^m a_k e_k \Bigr\|;
$$
for all $1 \leq n < m$ and all collections of scalars $(a_k)_{k=1}^m$. A basis (basic sequence) $(e_n)_{n=1}^\infty$ is said to be \emph{normalized} provided that $\|e_n\|=1$ for all $n$. If $(e_n)_{n=1}^\infty$ is a basic sequence then $(e_n/\|e_n\|)_{n=1}^\infty$ is a normalized basic sequence. Let $(e_n)_{n=1}^\infty$ be a basis sequence in $X$; $(a_n)_{n=1}^\infty$ a sequence of scalars and $0 \leq k_1 < k_2 < \ldots$ integers. A sequence $(u_n)_{n=1}^\infty$ of nonzero vectors in $X$ of the form
$$
u_n = \sum\limits_{i = k_n + 1}^{k_{n+1}} a_i e_i
$$
is called a \emph{block basis} of $(e_n)_{n=1}^\infty$. Every block basis (in particular, every subsequence) of a basic sequence is itself a basic sequence with a basis constant which does not exceed that of $(e_n)_{n=1}^\infty$.

\subsubsection{Unconditional bases}

A series $\sum_{n=1}^\infty x_n$ of elements of a Banach space $X$ is said to be \emph{unconditionally convergent} if for any permutation\footnote{e.i., a bijection} of the positive integers $\varphi: \mathbb N \to \mathbb N$ the series $\sum_{n=1}^\infty x_{\varphi (n)}$ converges. We need the next criterion of unconditional convergence \cite[Lemma~16.1]{S}.
\begin{lemma} \label{le:sin}
For any sequence $(x_n)_{n=1}^\infty$ in a Banach space $X$ the following assertions are equivalent
\begin{enumerate}
  \item[$(i)$] the series $\sum_{n=1}^\infty x_n$ unconditionally converges;
  \item[$(ii)$] for any sequence of signs $\theta_n = \pm 1$ the series $\sum_{n=1}^\infty \theta_n x_n$ converges;
  \item[$(iii)$] for any sequence of scalars $(a_n)_{n=1}^\infty$ such that $|a_n| \leq 1$, $n = 1,2, \ldots$ the series $\sum_{n=1}^\infty a_n x_n$ converges.
\end{enumerate}
\end{lemma}

A basis $(e_n)_{n=1}^\infty$ of a Banach space $X$ with the biorthogonal functionals $(e_n^*)_{n=1}^\infty$ is called an \emph{unconditional basis} if the series $\sum_{n=1}^\infty e_n^*(x) \, e_n$ converges unconditionally for every $x \in X$. In this case, for any subset $I \subseteq \mathbb N$ the projection $P_I x = \sum_{n \in I} e_n^*(x) \, e_n$ is well defined on $X$ and bounded, as well as for any sequence of signs $\Theta = (\theta_n)_{n=1}^\infty$, $\theta_n = \pm 1$ the operator $M_\Theta x = \sum_{n=1}^\infty \theta_n e_n^*(x) \, e_n$. Moreover, $\sup_I \|P_I\| \leq \sup_\Theta \|M_\Theta\| \leq 2 \sup_I \|P_I\| < \infty$ and the number $\sup_\Theta \|M_\Theta\|$ is called the \emph{unconditional constant} of the unconditional basis $(e_n)_{n=1}^\infty$. An unconditional basis with unconditional constant $1$ is said to be $1$-\emph{unconditional}. A basis which is not unconditional is called a \emph{conditional} basis. A sequence which is an unconditional (resp., conditional) basis in its closed linear span is called an \emph{unconditional basic sequence} (resp., \emph{conditional basic sequence}).

Every infinite dimensional Banach space contains a basic sequence, however, not every infinite dimensional separable Banach space contains a basis. The classical Banach spaces $L_1[0,1]$ and $C[0,1]$ contain bases, however they cannot be isomorphically embedded in a Banach space with an unconditional basis. The standard basis $e_n = (\underbrace{0, \ldots, 0}\limits_{n-1}, 1, 0, 0, \ldots)$ of the spaces $c_0$ and $\ell_p$ for $1 \leq p < \infty$ are $1$-unconditional.

We remark that every $1$-unconditional basic sequence $(e_n)_{n=1}^\infty$ in a Hilbert space $H$ is orthogonal, because the inequality $\|e_n + e_m\| = \|e_n - e_m\|$ yields
$$
(e_n,e_m) = \frac14 \Bigl( \|e_n + e_m\| - \|e_n - e_m\| \Bigr) = 0
$$
if $n \neq m$.

We also need the following statement from \cite[Proposition~1.c.7]{LTI} which is true for real Banach spaces.

\begin{lemma} \label{le:propLT}
Let $(e_n)_{n=1}^\infty$ be an unconditional basic sequence in a Banach space $X$ with the unconditional constant $M$. Let $(a_n)_{n=1}^\infty$ be any sequence of scalars for which the series $\sum_{n=1}^\infty a_n e_n$ converges. Then for any bounded sequence of scalars $(\lambda_n)_{n=1}^\infty$ one has
$$
\Bigl\| \sum_{n=1}^\infty \lambda_n a_n e_n \Bigr\| \leq M \sup_n |\lambda_n| \Bigl\| \sum_{n=1}^\infty a_n e_n \Bigr\|.
$$
\end{lemma}

Furthermore, we need the following finite dimensional version of Lemma~\ref{le:propLT}.

\begin{lemma} \label{le:propMMP}
Let $(x_n)_{n=1}^N$ be a finite sequence of elements in a real Banach space $X$. Then for any collection of scalars $(\lambda_n)_{n=1}^N$ one has
$$
\Bigl\| \sum_{n=1}^N \lambda_n x_n \Bigr\| \leq \max_n |\lambda_n| \max_{\theta_n = \pm 1} \Bigl\| \sum_{n=1}^N \theta_n x_n \Bigr\|.
$$
\end{lemma}

Formally Lemma~\ref{le:propMMP} does not follow from~\ref{le:propLT}, however its proof provided in the Lindenstrauss-Tzafriri book could be modified to prove Lemma~\ref{le:propMMP}. Besides, Lemma~\ref{le:propMMP} follows from Lemma~2.3 of \cite{MMP}.

\subsubsection{Boundedly complete bases}

A basis $(e_n)_{n=1}^\infty$ of a Banach space $X$ is called \emph{boundedly complete} if for any sequence of scalars $(a_n)_{n=1}^\infty$ the boundedness of the partial sums $\sup_n \bigl\| \sum_{k=1}^n a_k e_k \bigr\| < \infty$ implies the convergence of the series $\sum_{n=1}^\infty a_n e_n$. Every basis of a reflexive Banach space is boundedly complete \cite[Theorem~1.b.5]{LTI}. The standard basis of the nonreflexive space $\ell_1$ is evidently boundedly complete as well. However, every Banach space with a boundedly complete basis is isomorphic to a conjugate space \cite[Theorem~1.b.4]{LTI}. A kind of converse statement is also true: by a deep result of Johnson, Rosenthal and Zippin \cite{JRT}, if a conjugate Banach space $X^*$ has a basis then $X^*$ contains a boundedly complete basis. Finally, an unconditional basis of a Banach space $X$ is boundedly complete if and only if $X$ contains no subspace isomorphic to $c_0$ \cite[Theorem~1.c.10]{LTI}.

\subsection{A characterization of precompactness of sets in a Banach space with a basis}

We provide below a convenient characterization of precompactness in terms of biorthogonal functionals. Informally speaking, it asserts that the precompactness of a subset $A$ of a Banach space $X$ is equivalent to the uniform convergence of the Fourier series of elements of $A$ with respect to a given basis. Most likely, this statement is not new, however we do not know a citation, so we provide a complete proof.

\begin{lemma} \label{le:crcomp}
Let $X$ be a Banach space with a basis $(e_n)_{n=1}^\infty$ and the biorthogonal functionals $(e_n^*)_{n=1}^\infty$. A bounded set $A \subset X$ is precompact if and only if
\begin{equation} \label{eq:crcomp2}
\lim_{N \to \infty} \sup_{x \in A} \Bigl\| \sum_{n > N} e_n^*(x) \, e_n \Bigr\| = 0.
\end{equation}
\end{lemma}

\begin{proof}
Let $A$ be precompact. Assuming \eqref{eq:crcomp2} is false, we choose $\delta > 0$ so that
\begin{equation} \label{eq:crcomp3}
\limsup_{N \to \infty} \sup_{x \in A} \Bigl\| \sum_{n > N} e_n^*(x) \, e_n \Bigr\| > 2 \delta.
\end{equation}

Then we construct a block basis $u_k = \sum_{n=n_k+1}^{n_{k+1}} a_n e_n$, $a_k \in \mathbb R$, $0 \leq n_1 < n_2 < \ldots$ and a sequence $x_k \in A$ so that
\begin{equation} \label{eq:cond}
\|u_k\| \geq \delta \, \frac{4K+1}{2K+1} \,\,\,\,\, \mbox{и} \,\, \|x_k - u_k\| \leq \frac{\delta}{2K+1} \, \,\,\, \mbox{при} \,\, k = 1,2, \ldots,
\end{equation}
where $K$ is the basis constant of $(e_n)_{n=1}^\infty$. Choose by \eqref{eq:crcomp3} a number $n_1 \geq 0$ and $x_1 \in A$ so that $\bigl\| \sum_{n > n_1} e_n^*(x_1) \, e_n \bigr\| \geq 2 \delta.$ Then pick a number $m_1 > n_1$ such that $\bigl\| \sum_{n > m_1} e_n^*(x_1) \, e_n \bigr\| < \frac{\delta}{2K+1} \, .$ Then for $u_1 = \sum_{n = n_1+1}^{m_1} e_n^*(x_1) \, e_n$ we obtain $\|x_1 - u_1\| < \frac{\delta}{2K+1}$ and
$$
\|u_1\| \geq \Bigl\| \sum_{n > n_1} e_n^*(x_1) \, e_n \Bigr\| - \Bigl\| \sum_{n > m_1} e_n^*(x_1) \, e_n \Bigr\| \geq 2 \delta - \frac{\delta}{2K+1} = \delta \, \frac{4K+1}{2K+1} \, .
$$

On the second step, choose by \eqref{eq:crcomp3} an integer $n_2 > m_1$ and $x_2 \in A$ so that $\bigl\| \sum_{n > n_2} e_n^*(x_2) \, e_n \bigr\| \geq 2 \delta,$ and then choose $m_2 > n_2$ so that $\bigl\| \sum_{n > m_2} e_n^*(x_2) \, e_n \bigr\| < \frac{\delta}{2K+1} \, .$ Then for $u_2 = \sum_{n = n_2+1}^{m_2} e_n^*(x_2) \, e_n$ we get $\|x_2 - u_2\| < \frac{\delta}{2K+1}$ and
$$
\|u_2\| \geq \Bigl\| \sum_{n > n_1} e_n^*(x_2) \, e_n \Bigr\| - \Bigl\| \sum_{n > n_2} e_n^*(x_2) \, e_n \Bigr\| \geq 2 \delta - \frac{\delta}{2K+1} = \delta \, \frac{4K+1}{2K+1} \, .
$$

Continuing analogously the recursive procedure of choice for $k = 3,4, \ldots$, we obtain a block basis $u_k = \sum_{n=n_k+1}^{n_{k+1}} a_n e_n$ and a sequence $x_k \in A$ for which \eqref{eq:cond} holds.

Fix any integers $k < m$ and observe that, by \eqref{eq:cond},
$$
\delta \, \frac{4K+1}{2K+1} \leq \|u_k\| \leq K \|u_k - u_m\|.
$$

Hence,
\begin{align*}
\|x_k - x_m\| &\geq \|u_k - u_m\| - \|x_k - u_k\| - \|x_m - u_m\|\\
&\geq \frac{\delta}{K} \, \frac{4K+1}{2K+1} - \frac{\delta}{2K+1} - \frac{\delta}{2K+1}\\
&= \Bigl( 2 + \frac1K \Bigr) \, \frac{\delta}{2K+1} = \alpha > 0.
\end{align*}
Finally the sequence $(x_k)_{k=1}^\infty$ in $A$ appears to be $\alpha$-separate, which contradicts the precompactness of $A$. So, \eqref{eq:crcomp2} is proved.

Let for a bounded set $A$ condition \eqref{eq:crcomp2} hold. We prove that $A$ is precompact. Fix any $\varepsilon > 0$ and construct a finite $\varepsilon$-net in $A$. Choose by \eqref{eq:crcomp2} a number $N \in \mathbb N$ such that $\bigl\| \sum_{n > N} e_n^*(x) \, e_n \bigr\| \leq \varepsilon/3$ for all $x \in A$. Denote by $P_N$ the basis projection of $X$ defined by $P_N x = \sum_{n=1}^N e_n^*(x) \, e_n$ for all $x \in X$. Since the image $P_N(A)$ is a bounded set in a finite dimensional normed space, there is a finite $\varepsilon/3$-net $P_N(x_1), \ldots, P_N(x_m)$ for $P_N(A)$, where $x_1, \ldots, x_m \in A$. Show that $x_1, \ldots, x_m \in A$ is an $\varepsilon$-net for $A$. Indeed, let $x \in A$. Choose $k \in \{1, \ldots, m\}$ so that $\|P_N x - P_N x_k \| \leq \varepsilon/3$. Then
\begin{align*}
\|x - x_k\| &\leq \|P_N (x - x_k)\| + \|x - P_N x\| + \|x_k - P_N x_k\|\\
&=\|P_N x - P_N x_k\| + \Bigl\| \sum_{n > N} e_n^*(x) \, e_n \Bigr\| + \Bigl\| \sum_{n > N} e_n^*(x_k) \, e_n \Bigr\|\\
&\leq \frac{\varepsilon}{3} + \frac{\varepsilon}{3} + \frac{\varepsilon}{3} = \varepsilon.
\end{align*}
\end{proof}

\section{Bricks}

The notion of entropy is based on the concept of bricks in a Banach space, that is, a box with sides that are parallel to the coordinate hyperplanes with respect to a given basis. The latter concept we develop in this section.

\subsection{Definition and properties}

Let $X$ be a Banach space with a basis $\mathcal B = (e_n)_{n=1}^\infty$ and biorthogonal functionals $e_k^* \in X_0^*$, and let $\mathcal E = (\varepsilon_n)_{n=1}^\infty$ be a sequence of nonnegative numbers.

\begin{definition}
A \emph{brick} (more precisely, the \emph{brick corresponding to the pair} $(\mathcal B, \mathcal E)$) is defined to be the following set
$$
K_{\mathcal B, \mathcal E} = \bigl\{ x \in X_0: \,\, (\forall n \in \mathbb N) \,\, |e_n^*(x)| \leq \varepsilon_n \bigr\}.
$$

The numbers $\varepsilon_n$ are called the \emph{half-hight} of the brick $K_{\mathcal B, \mathcal E}$.
\end{definition}

In other words, $K_{\mathcal B, \mathcal E}$ consists of all sums of convergent series $x = \sum_{n=1}^\infty a_n e_n$ with coefficients satisfying $|a_n| \leq \varepsilon_n$ for all $n$.

A simple observation: any brick $K_{\mathcal B, \mathcal E}$ coincides with the brick $K_{\mathcal B', \mathcal E'}$, where $\mathcal B' = (e_n')_{n \in M}$ is the normalized basis $e_n' = \|e_n\|^{-1} e_n$, $n = 1,2, \ldots$, and $\mathcal B' = (e_n')_{n \in M}$ the half-height $\varepsilon_n' = \varepsilon_n \|e_n\|$, $n = 1,2, \ldots$. So, \textbf{we consider bricks constructed by normalized bases only.}

\begin{definition}
A brick constructed by an unconditional basis, $1$-unconditional basis, or boundedly complete basis is called an \emph{unconditional, $1$-unconditional} or respectively, a \emph{boundedly complete} brick.
\end{definition}

Recall that a subset $A$ of a linear space $X$ is called \emph{absolutely convex} provided for all $m \in \mathbb N$, $x_1, \ldots, x_m \in A$ and $\lambda_1, \ldots, \lambda_m \in \mathbb K$ the inequality $|\lambda_1| + \ldots + |\lambda_m| \leq 1$ implies $\lambda_1 x_1 + \ldots + \lambda_m x_m \in A$.

\begin{prop} \label{pr:closed}
Every brick in a Banach space $X$ is an absolutely convex closed subset of $X$.
\end{prop}

\begin{proof}
Let $\mathcal B = (e_n)_{n=1}^\infty$ be a normalized basis of $X$ with the biorthogonal functionals $e_k^* \in X^*$ and $\mathcal E = (\varepsilon_n)_{n=1}^\infty$. Assume $m \in \mathbb N$, $x_1, \ldots, x_m \in K_{\mathcal B, \mathcal E}$, $\lambda_1, \ldots, \lambda_m \in \mathbb K$ and $|\lambda_1| + \ldots + |\lambda_m| \leq 1$. Then for every $n \in \mathbb N$ one has
\begin{align*}
|e_n^*(\lambda_1 x_1 + \ldots + \lambda_m x_m)| &\leq |\lambda_1| |e_n^*(x_1)| + \ldots + |\lambda_m | |e_n^*(x_m)|\\
&\leq |\lambda_1| \, \varepsilon_n + \ldots + |\lambda_m | \, \varepsilon_n \leq \varepsilon_n.
\end{align*}
Thus, $K_{\mathcal B, \mathcal E}$ is absolutely convex. By continuity of $e_n^*$'s, $K_{\mathcal B, \mathcal E}$ is closed.
\end{proof}

One can deduce from Lemma~\ref{le:crcomp} that if $K_{\mathcal B, \mathcal E}$ is compact then $\lim_{n \to \infty} \varepsilon_n = 0$, and also if $\sum_{n=1}^\infty \varepsilon_n < \infty$, then $K_{\mathcal B, \mathcal E}$ is compact. We are not going to provide details because of the more general characterization of compactness for bricks below (Theorem~\ref{th:crcomp}).

\begin{definition} \label{def:solid}
A brick $K_{\mathcal B, \mathcal E}$ is said to be \emph{solid} if for each $x \in K_{\mathcal B, \mathcal E}$ and each numbers $a_1, a_2, \ldots \in \mathbb K$ such that $|a_n| \leq |e_n^*(x)|$ for all $n \in \mathbb N$ the series $\sum_{n=1}^\infty a_n e_n$ converges\footnote{and hence, its sum belongs to $K_{\mathcal B, \mathcal E}$}.
\end{definition}

\begin{prop} \label{pr:teles}
Let $X$ be a Banach space with a normalized basis $\mathcal B = (e_n)_{n=1}^\infty$ and biorthogonal functionals $e_k^* \in X_0^*$, and let $\mathcal E = (\varepsilon_n)_{n=1}^\infty$. Then each of the following assertions is sufficient for $K_{\mathcal B, \mathcal E}$ to be solid
\begin{enumerate}
  \item $K_{\mathcal B, \mathcal E}$ is unconditional;
  \item $\sum_{n=1}^\infty \varepsilon_n < \infty$.
\end{enumerate}
\end{prop}

\begin{proof}
(1) follows directly from the definitions and (2) follows from the inequality
$$
\Bigl\| \sum_{k=n+1}^{n+l} a_k e_k \Bigr\| \leq \sum_{k=n+1}^{n+l} |a_k| \|e_k\| \leq \sum_{k=n+1}^{n+l} |e_k^*(x)| \|e_k\| \leq \sum_{k=n+1}^{n+l} \varepsilon_k \|e_k\|.
$$
\end{proof}

Since conditions (1) and (2) of Proposition~\ref{pr:teles} do not imply each other, neither of them is necessary for the brick to be solid.

Recall that an element $x_0 \in A$ of a subset $A$ of a linear space $X$ is called an \emph{extreme point} of $A$ if there is no segment of $A$ centered at $x_0$, e.i. for every $x \in X$ there exists $\lambda \in [-1,1]$ such that $x_0 + \lambda x \notin A$. The next statement easily follows from the definitions.

\begin{prop}
Let $X$ be a Banach space with a normalized basis $\mathcal B = (e_n)_{n=1}^\infty$ and biorthogonal functionals $e_k^* \in X_0^*$, and let $\mathcal E = (\varepsilon_n)_{n=1}^\infty$. A vector $x_0 \in K_{\mathcal B, \mathcal E}$ is an extreme point of $K_{\mathcal B, \mathcal E}$ if and only if $|e_n^*(x_0)| = \varepsilon_n$ for all $n \in \mathbb N$.
\end{prop}

It is immediate that if $K_{\mathcal B, \mathcal E}$ has an extreme point then $\lim_{n \to \infty} \varepsilon_n = 0$. As we will see below, the existence of an extreme point of a brick unrelated to its boundedness.

Observe that every element $x_0 \in X$ generates a brick corresponding to a normalized basis $\mathcal B = (e_n)_{n=1}^\infty$ of $X$ and the sequence $\varepsilon_n = |e_n^*(x_0)|$, an extreme point of which $x_0$ is.

First we show that the existence of an extreme point of a brick does not imply its boundedness.

\begin{example} \label{ex:unboundedwithre}
There exists an unbounded brick with en extreme point.
\end{example}

\begin{proof}
Let $X = c$ be the space of all converging sequences with the supremum norm. Consider the summing basis \cite[p.~20]{LTI} $e_n = (\underbrace{0, \ldots, 0}_{n-1}, 1, 1, \ldots)$, and the brick generated by the element
$$
x_0 = e_1 - \frac{e_2}{2} + \frac{e_3}{3} - \ldots + \frac{(-1)^{n+1}e_n}{n} + \ldots.
$$
The convergence of the series in $c$ follows from that of the Leibniz series $\sum_{n=1}^\infty \frac{(-1)^{n+1}}{n}$, and the unboundedness of the brick $K_{\mathcal B, \mathcal E}$ with the half-height $\varepsilon_n = \frac{1}{n}$ is guaranteed by the equality
$$
\Bigl\| e_1 + \frac{e_2}{2} + \frac{e_3}{3} + \ldots + \frac{e_n}{n} \Bigr\| = \sum_{k=1}^n \frac1k
$$
and the divergence of the harmonic series.
\end{proof}

To the contrast, every unconditional brick with an extreme point is bounded. Moreover, the norm of any extreme point (and hence, of an arbitrary element) is estimated by the unconditional constant of the basis and the norm of any fixed extreme point.

\begin{prop} \label{pr:mmxnc}
Let $X$ be a Banach space with a normalized unconditional basis $\mathcal B = (e_n)_{n=1}^\infty$ and the biorthogonal functionals $e_k^* \in X_0^*$ and the unconditional constant $M$, $\mathcal E = (\varepsilon_n)_{n=1}^\infty$. Let $x_0$ be an extreme point of the brick $K_{\mathcal B, \mathcal E}$. Then $K_{\mathcal B, \mathcal E}$ is bounded by $M \, \|x_0\|.$ In particular, if $\mathcal B$ is $1$-unconditional then $\|x\| \leq \|x_0\|$ for every $x \in K_{\mathcal B, \mathcal E}$.
\end{prop}

Proposition~\ref{pr:mmxnc} follows from Lemma~\ref{le:propLT}.

Now we show that the boundedness of a brick does not imply the existence of an extreme point, even of a $1$-unconditional brick.

\begin{example} \label{ex:boundedwithoutre}
There exists a bounded $1$-unconditional brick without an extreme point.
\end{example}

\begin{proof}
Such a brick, for example, is the closed unit ball of the space $c_0$. Indeed, consider the $1$-unconditional standard basis of $c_0$ and take $\varepsilon_n = 1$, $n = 1,2, \ldots$. Then the absence of extreme points is obvious.
\end{proof}

\begin{prop}
Every bounded boundedly complete brick contains an extreme point.
\end{prop}

\begin{proof}
Let $X$ be a Banach space with a normalized boundedly complete basis $\mathcal B = (e_n)_{n=1}^\infty$ and biorthogonal functionals $e_k^* \in X_0^*$, and let $\mathcal E = (\varepsilon_n)_{n=1}^\infty$. Assume $\|x\| \leq L$ for all $x \in K_{\mathcal B, \mathcal E}$ and some number $L$. Observe that for $a_n = \varepsilon_n$ the condition $\bigl\| \sum_{k=1}^n a_k e_k \bigr\| \leq L$ holds for every $n \in \mathbb N$, because $\sum_{k=1}^n a_k e_k \in K_{\mathcal B, \mathcal E}$. Since the basis is boundedly complete, the series $x_0 = \sum_{n=1}^\infty \varepsilon_n e_n$ converges, and hence there is an extreme point $x_0$.
\end{proof}

\section{Radii and a characterization of the compactness for bricks}

We consider the following three radii of a brick: the extreme radius, the unconditional radius and the absolute radius. In the case where the unconditional radius of a brick is finite all the three radii coincide. Moreover, in this case (and only in this case) the brick is compact. If an extreme radius is finite then it equals the absolute radius.

Let $X$ be a Banach space with a normalized basis $\mathcal B = (e_n)_{n=1}^\infty$ and biorthogonal functionals $e_k^* \in X_0^*$, and let $\mathcal E = (\varepsilon_n)_{n=1}^\infty$.

\begin{definition} \label{def:radii}
The \emph{extreme radius} $r^{\rm ext}(K_{\mathcal B, \mathcal E})$, the \emph{unconditional radius} $r^{\rm unc}(K_{\mathcal B, \mathcal E})$ and the \emph{absolute radius} $\sup\limits_{x \in K_{\mathcal B, \mathcal E}} \|x\|$ of the brick $K_{\mathcal B, \mathcal E}$ is defined to be either a number or a symbol $\infty$ as follows.
\begin{enumerate}
  \item $r^{\rm ext}(K_{\mathcal B, \mathcal E}) = \sup \bigl\{\|x_0\|: \, x_0 \,\, \mbox{\footnotesize is an extreme point of} \,\, K_{\mathcal B, \mathcal E} \bigr\}$, if an extreme point exists, and $r^{\rm ext}(K_{\mathcal B, \mathcal E}) = \infty$ otherwise.
  \item $\displaystyle{r^{\rm unc}(K_{\mathcal B, \mathcal E}) = \sup_{\theta_n = \pm 1} \Bigl\| \sum_{n=1}^\infty \theta_n \varepsilon_n e_n \Bigr\|}$ (the norm of a divergent series is $\infty$).
  \item $\sup\limits_{x \in K_{\mathcal B, \mathcal E}} \|x\|$.
\end{enumerate}
\end{definition}

We do not offer a special symbol for the absolute radius, because the formula in (3) is not involved and clear. The difference between the defined radii could be demonstrated using Example~\ref{ex:boundedwithoutre} where as a brick we take the closed unit ball $B_{c_0}$ of the space $c_0$. By the definitions, $r^{\rm unc}(B_{c_0}) = 1$, however $r^{\rm ext}(B_{c_0}) = r^{\rm unc}(B_{c_0}) = \infty$. Below we construct a brick $K_{\mathcal B, \mathcal E}$ (Example~\ref{ex:uncompact}), for which $r^{\rm ext}(K_{\mathcal B, \mathcal E}) = \sup\limits_{x \in K_{\mathcal B, \mathcal E}} \|x\| < \infty$, however $r^{\rm unc}(B_{c_0}) = \infty$. On the other hand, Example~\ref{ex:unboundedwithre} may mislead the reader by hinting that a brick with a finite extreme radius need not be bounded. Actually, we have the following statement on the connection between the radii.

\subsection{The connection between radii}

\begin{theorem} \label{pr:kirprad}
For an arbitrary brick $K_{\mathcal B, \mathcal E}$ in a Banach space $X$ the following assertions hold.
\begin{enumerate}
  \item $r^{\rm ext}(K_{\mathcal B, \mathcal E}) \leq r^{\rm unc}(K_{\mathcal B, \mathcal E})$.
  \item If $r^{\rm ext}(K_{\mathcal B, \mathcal E}) < \infty$ then $r^{\rm ext}(K_{\mathcal B, \mathcal E}) = \sup\limits_{x \in K_{\mathcal B, \mathcal E}} \|x\|$.
  \item If $r^{\rm unc}(K_{\mathcal B, \mathcal E}) < \infty$ then $r^{\rm ext}(K_{\mathcal B, \mathcal E}) = r^{\rm unc}(K_{\mathcal B, \mathcal E}) = \sup\limits_{x \in K_{\mathcal B, \mathcal E}} \|x\|$.
\end{enumerate}
\end{theorem}

\begin{proof}
Item (1) follows immediately from the definitions.

(2) Assume $r^{\rm ext}(K_{\mathcal B, \mathcal E}) < \infty$. By the definitions, $r^{\rm ext}(K_{\mathcal B, \mathcal E}) \leq \sup\limits_{x \in K_{\mathcal B, \mathcal E}} \|x\|$. The inequality $\sup\limits_{x \in K_{\mathcal B, \mathcal E}} \|x\| \leq r^{\rm ext}(K_{\mathcal B, \mathcal E})$ is quite thin; its proof we present separately (see Lemma~\ref{th:bound} below).

(3) Assume $r^{\rm unc}(K_{\mathcal B, \mathcal E}) < \infty$. The equality $r^{\rm ext}(K_{\mathcal B, \mathcal E}) = r^{\rm unc}(K_{\mathcal B, \mathcal E})$ follows from the definitions as well, and the equality $\sup\limits_{x \in K_{\mathcal B, \mathcal E}} \|x\| = r^{\rm ext}(K_{\mathcal B, \mathcal E})$ follows from item (2).
\end{proof}

\begin{lemma} \label{th:bound}
If $r^{\rm ext}(K_{\mathcal B, \mathcal E}) < \infty$ then the brick $K_{\mathcal B, \mathcal E}$ is bounded by $r^{\rm ext}(K_{\mathcal B, \mathcal E})$.
\end{lemma}

\begin{proof}
First we prove that for every $n_0 \in \mathbb N$ and every $\varepsilon > 0$ there is $N \geq n_0$ such that for all signs $\theta_1, \ldots, \theta_N \in \{-1,1\}$ one has
\begin{equation} \label{ffk}
\Bigl\| \sum_{n=1}^N \theta_n \varepsilon_n e_n \Bigr\| < r^{\rm ext}(K_{\mathcal B, \mathcal E}) + \varepsilon.
\end{equation}

Indeed, fix any extreme point $x_0 = \sum_{n=1}^\infty \alpha_n e_n$ of $K_{\mathcal B, \mathcal E}$, $|\alpha_n| = \varepsilon_n$, $n = 1,2, \ldots$ (an extreme point exists because $r^{\rm ext}(K_{\mathcal B, \mathcal E}) < \infty$). Choose $N \geq n_0$ so that
\begin{equation} \label{ffk2}
\Bigl\| \sum_{n>N} \alpha_n e_n \Bigr\| < \varepsilon.
\end{equation}

Let $\theta_1, \ldots, \theta_N \in \{-1,1\}$ be any signs. Observe that
$$
x = \sum_{n=1}^N \theta_n \varepsilon_n e_n + \sum_{n>N} \alpha_n e_n
$$
is an extreme point of $K_{\mathcal B, \mathcal E}$, hence $\|x\| \leq r^{\rm ext}(K_{\mathcal B, \mathcal E})$. Taking into account \eqref{ffk2}, we obtain
$$
\Bigl\| \sum_{n=1}^N \theta_n \varepsilon_n e_n \Bigr\| \leq \|x\| + \Bigl\| \sum_{n>N} \alpha_n e_n \Bigr\| < r^{\rm ext}(K_{\mathcal B, \mathcal E}) + \varepsilon.
$$
Thus, \eqref{ffk} is proved.

Let $\hat{x} \in K_{\mathcal B, \mathcal E}$ be any element. Show that $\|\hat{x}\| \leq r^{\rm ext}(K_{\mathcal B, \mathcal E})$. Fix any $\varepsilon > 0$ and pick $n_0 \in \mathbb N$ so that for each $m \geq n_0$
\begin{equation} \label{ffk3}
\Bigl\| \sum_{n>m} e_n^*(\hat{x}) \, e_n \Bigr\| < \varepsilon.
\end{equation}

Then by the above, choose $N \geq n_0$ so that for all signs $\theta_1, \ldots, \theta_N \in \{-1,1\}$ one has \eqref{ffk}. Then
\begin{align*}
\|\hat{x}\| = \Bigl\| \sum_{n=1}^\infty e_n^*(\hat{x}) \, e_n \Bigr\| &\stackrel{\mbox{\tiny \eqref{ffk3}}}{\leq} \Bigl\| \sum_{n=1}^N e_n^*(\hat{x}) \, e_n \Bigr\| + \varepsilon\\
&\stackrel{\mbox{\tiny Lemma~\ref{le:propMMP}}}{\leq} \max_{\theta_n = \pm 1} \Bigl\| \sum_{n=1}^N \theta_n \varepsilon_n e_n \Bigr\| + \varepsilon\\
&\stackrel{\mbox{\tiny \eqref{ffk}}}{<} r^{\rm ext}(K_{\mathcal B, \mathcal E}) + 2 \varepsilon.
\end{align*}

By the arbitrariness of $\varepsilon > 0$, $\|\hat{x}\| \leq r^{\rm ext}(K_{\mathcal B, \mathcal E})$. Thus, $K_{\mathcal B, \mathcal E}$ is bounded by $r^{\rm ext}(K_{\mathcal B, \mathcal E})$.
\end{proof}

\subsection{Bricks with finite extreme radius}

In this subsection we study the question of the compactness of a brick with finite extreme radius.

\begin{example} \label{ex:uncompact}
There exists a noncompact brick of finite extreme radius.
\end{example}

\begin{proof}
This example is a modification of Example~\ref{ex:boundedwithoutre}. We choose integers $0 = n_0$,  $2 = n_1 < n_2 < \ldots$ such that
\begin{equation} \label{eq:kkkkdf}
\frac{1}{n_{k-1} + 1} + \ldots + \frac{1}{n_k} \in [1,2], \,\,\, k = 1,2, \ldots.
\end{equation}

Then for $X = c_0$ we define a basis $\mathcal B = (f_n)_{n=1}^\infty$ by
\begin{align*}
&f_1 = (1,1,0,0, \ldots),\\
&f_2 = (0,1,0,0, \ldots),\\
&f_3 = (0,0,\underbrace{1,1, \ldots, 1, 1}\limits_{n_2-n_1}, 0,0, \ldots),\\
&f_4 = (0,0,\underbrace{0,1, \ldots, 1, 1}\limits_{n_2-n_1}, 0,0, \ldots),\\
&\ldots\\
&f_{n_2} = (0,0,\underbrace{0,0, \ldots, 0, 1}\limits_{n_2-n_1}, 0,0, \ldots),\\
&f_{n_2+1} = (\underbrace{0,0, \ldots, 0}\limits_{n_2}, \underbrace{1,1, \ldots, 1, 1}\limits_{n_3-n_2}, 0,0, \ldots),\\
&\ldots .\\
\end{align*}

Using standard arguments, one can prove that the above system $\mathcal B = (f_n)_{n=1}^\infty$ is a basis of $c_0$. More precisely, first we need to prove the inequality $\bigl\| \sum_{k=1}^n a_k f_k \bigr\| \leq \bigl\| \sum_{k=1}^m a_k f_k \bigr\|$ for all $n < m$ and any collection of scalars $(a_k)_{k=1}^m$. Then we prove that the linear span of $(f_n)_{n=1}^\infty$ is dense in $c_0$ (because the standard basis of $c_0$ is contained in that linear span). We omit the details which are straightforward.

Then we define half-height by $\varepsilon_n = \frac1n$, $n = 1,2, \ldots$, set $\mathcal E = (\varepsilon_n)_{n=1}^\infty$ and prove that the brick $K_{\mathcal B, \mathcal E}$ is as desired. First we show that $K_{\mathcal B, \mathcal E}$ contains an extreme point. Indeed, the series $f_0 = \sum_{n=1}^\infty (-1)^{n+1} \varepsilon_n f_n$ converges in $c_0$, because the Leibniz series $\sum_{n=1}^\infty \frac{(-1)^{n+1}}{n}$ converges, and hence, possesses the Cauchy condition. By \eqref{eq:kkkkdf} we obtain that the brick is norm bounded by $2$, hence, $r^{\rm ext}(K_{\mathcal B, \mathcal E}) \leq  2$. The noncompactness of $K_{\mathcal B, \mathcal E}$ follows from the fact that the sequence
$$
g_k = \frac{1}{n_{k-1} + 1} f_{n_{k-1} + 1} + \ldots + \frac{1}{n_k} f_{n_k}, \,\,\, k = 1,2, \ldots.
$$
satisfies $g_k \in K_{\mathcal B, \mathcal E}$ and $\|g_k\| \geq 1$ by \eqref{eq:kkkkdf}.
\end{proof}

Now we show that in the most natural cases (unconditional or boundedly complete basis) a brick of finite extreme radius is compact.

\begin{theorem} \label{th:comp}
Every unconditional or boundedly complete brick of finite extreme radius is compact.
\end{theorem}

\begin{proof}
Let $X$ be a Banach space with a normalized unconditional or boundedly complete basis $\mathcal B = (e_n)_{n=1}^\infty$ and biorthogonal functionals $e_k^* \in X_0^*$, and let $\mathcal E = (\varepsilon_n)_{n=1}^\infty$. Assume $r^{\rm ext}(K_{\mathcal B, \mathcal E}) < \infty$, and show that for $A = K_{\mathcal B, \mathcal E}$ we have \eqref{eq:crcomp2}.

\emph{The case where $\mathcal B$ is unconditional.} Let $M$ be the unconditional constant of $\mathcal B$ and $x_0$ any extreme point of $A$ (an extreme point exists, because $r^{\rm ext}(K_{\mathcal B, \mathcal E}) < \infty$). Fix any $\varepsilon > 0$ and choose $n_0 \in \mathbb N$ so that for every $N \geq n_0$ one has $\bigl\| \sum_{n>N} e_n^*(x_0) \, e_n \bigr\| < M^{-1} \varepsilon$. Then for each $x \in K_{\mathcal B, \mathcal E}$ and each $N \geq n_0$, taking into account $|e_n^*(x)| \leq \varepsilon_n = |e_n^*(x_0)|$ and Lemma~\ref{le:propLT}, we obtain
\begin{align*}
\Bigl\| \sum_{n>N} e_n^*(x) \, e_n \Bigr\| \leq M \Bigl\| \sum_{n>N} e_n^*(x_0) \, e_n \Bigr\| < \varepsilon.
\end{align*}

\emph{The case where $\mathcal B$ is boundedly complete.} Assume \eqref{eq:crcomp2} is false. Choose $\delta > 0$ and a sequence $x_{N_k} \in K_{\mathcal B, \mathcal E}$ so that $N_1 < N_2 < \ldots$, $\|x_{N_k}\| \geq 2 \delta$ and $e_n^*(x_{N_k}) = 0$ as $n \leq N_k$, that is, $x_{N_k} = \sum_{n > N_k} e_n^*(x_{N_k}) \, e_n$ for $k = 1,2, \ldots$. We are going to construct a block basis $(u_k)_{k=1}^\infty$ of $\mathcal B$ such that $u_k \in K_{\mathcal B, \mathcal E}$ and $\|u_k\| \geq \delta$ for $k = 1,2, \ldots$. Set $n_1 = 0$ and choose $n_2 > n_1$ so that
$$
\Bigl\| \sum_{n>n_2} e_n^*(x_{N_1}) \, e_n \Bigr\| < \delta.
$$
Then for $u_1 = \sum_{n=1}^{n_2} e_n^*(x_{N_1}) \, e_n$ one gets that $u_1 \in K_{\mathcal B, \mathcal E}$ and
$$
\|u_1\| \geq \bigl\| \sum_{n=1}^\infty e_n^*(x_{N_1}) \, e_n - \sum_{n>{n_2}} e_n^*(x_{N_1}) \, e_n \Bigr\| \geq 2 \delta - \delta = \delta.
$$

At the second step we choose $j_2 > j_1 = 1$ so that $N_{j_2} > n_2$. Thus,
$$
x_{N_{j_2}} = \sum_{n > N_{j_2}} e_n^*(x_{N_{j_2}}) \, e_n = \sum_{n > n_2} e_n^*(x_{N_{j_2}}) \, e_n.
$$

Choose $n_3 > n_2$ so that
$$
\Bigl\| \sum_{n>n_3} e_n^*(x_{N_{j_2}}) \, e_n \Bigr\| < \delta.
$$
Then for $u_2 = \sum_{n=n_2+1}^{n_3} e_n^*(x_{N_{j_2}}) \, e_n$ we obtain that $u_2 \in K_{\mathcal B, \mathcal E}$ and
$$
\|u_2\| \geq \bigl\| \sum_{n=n_2+1}^\infty e_n^*(x_{N_{j_2}}) \, e_n - \sum_{n>{n_3}} e_n^*(x_{N_{j_2}}) \, e_n \Bigr\| \geq 2 \delta - \delta = \delta.
$$

Proceeding like that step by step, we construct the desired block basis $(u_k)_{k=1}^\infty$. Now for each $n \leq N_1$ set $a_n = e_n^*(x_{N_1})$, and for every $N_{j_k} < n \leq N_{j_{k+1}}$ set $a_n = e_n^*(x_{N_{j_k}})$, $k = 2,3, \ldots$. Then
$$
u_1 = \sum_{n=1}^{N_1} a_n e_n \,\,\,\,\, \mbox{and} \,\,\, u_k = \sum_{n=N_{j_k}+1}^{N_{j_{k+1}}} a_n e_n \,\,\, \mbox{for} \,\,\, k = 2,3, \ldots.
$$

Since $S_n = \sum_{i=1}^n a_i e_i \in K_{\mathcal B, \mathcal E}$ for all $n \in \mathbb N$, we get $\|S_n\| \leq r^{\rm ext}(K_{\mathcal B, \mathcal E})$ by Theorem~\ref{th:bound}. Since the basis $(e_n)_{n=1}^\infty$ is boundedly complete, the series $\sum_{n=1}^\infty a_n e_n$ converges. However, this is impossible, because the Cauchy condition for its convergence contradicts the inequalities $\|u_k\| \geq \delta$, $k = 1,2, \ldots$. So, \eqref{eq:crcomp2} is valid.

Thus, \eqref{eq:crcomp2} holds anyway. By Lemma~\ref{le:crcomp}, the brick $K_{\mathcal B, \mathcal E}$ is precompact. By Proposition~\ref{pr:closed}, $K_{\mathcal B, \mathcal E}$ is compact.
\end{proof}

\subsection{A characterization of the compactness for bricks}

In this subsection we characterize the compactness for bricks, partially in terms of the following notion.

\begin{definition}
Let $X$ be a Banach space with a normalized basis $\mathcal B = (e_n)_{n=1}^\infty$ and biorthogonal functionals $e_k^* \in X_0^*$, and let $\mathcal E = (\varepsilon_n)_{n=1}^\infty$. The brick $K_{\mathcal B, \mathcal E}$ is called \emph{holistic} if for any sequence of scalars $(a_n)_{n=1}^\infty$ such that $|a_n| \leq \varepsilon_n$ for all $n \in \mathbb N$ the series $\sum_{n=1}^\infty a_n e_n$ converges.
\end{definition}

In other words, a holistic brick is a solid brick with an extreme point (cf. Definition~\ref{def:solid}).

The following result is important for the concept of entropy.

\begin{theorem}[The compactness test for bricks] \label{th:crcomp}
Let $X$ be a Banach space with a normalized basis $\mathcal B = (e_n)_{n=1}^\infty$ and biorthogonal functionals $e_k^* \in X_0^*$, and let $\mathcal E = (\varepsilon_n)_{n=1}^\infty$. Then the following assertions are equivalent.
\begin{enumerate}
  \item The brick $K_{\mathcal B, \mathcal E}$ is compact.
  \item The brick $K_{\mathcal B, \mathcal E}$ is holistic.
  \item The series $\sum_{n=1}^\infty \varepsilon_n e_n$ converges unconditionally.
  \item $r^{\rm unc}(K_{\mathcal B, \mathcal E}) < \infty$.
\end{enumerate}
\end{theorem}

\begin{proof}
The equivalence (2) $\Leftrightarrow$ (3) follows from Lemma~\ref{le:sin}.

(1) $\Rightarrow$ (2). Let $(a_n)_{n=1}^\infty$ be a sequence of scalars such that $|a_n| \leq \varepsilon_n$ for all $n \in \mathbb N$. Set $x_N = \sum_{n=1}^N a_n e_n$ and show that the series $\sum_{n=1}^\infty a_n e_n$ converges. Assume, on the contrary, that this is false. Then the series does not meet the Cauchy condition, and hence, there are $\delta > 0$, sequences of integers $1 = n_0 < n_1 < \ldots$ and $(\ell_k)_{k=1}^\infty$ such that $\ell_k \leq n_{k+1} - n_k$ and $\|u_k\| \geq \delta$ for $k = 1,2, \ldots$, where $u_k = \sum_{j = n_k+1}^{n_k+\ell} a_j e_j$. Observe that $u_k \in K_{\mathcal B, \mathcal E}$ for all $k$. Denote by $K$ the basis constant of $\mathcal B$ and prove that $\|u_k - u_m\| \geq \delta/K$ for all $k < m$, which contradicts the compactness of $K_{\mathcal B, \mathcal E}$.Indeed, $\delta \leq \|u_k\| \leq K \|u_k - u_m\|$.

(2) $\Rightarrow$ (1). We prove that
\begin{equation} \label{eq:ccbvn}
\lim_{N \to \infty} \sup_{|a_n| \leq \varepsilon_n} \Bigl\| \sum_{n > N} a_n e_n \Bigr\| = 0.
\end{equation}

Indeed, if this were false, we would choose $\delta > 0$, последовательность номеров $0 = n_0 < n_1 < \ldots$ and a sequence $(a_n)_{n=1}^\infty$, $|a_n| \leq \varepsilon_n$ so that
$$
\Bigl\| \sum_{j = n_{k-1} +1}^{n_k} a_j e_j \Bigr\| \geq \delta,
$$
which contradicts the Cauchy condition for $\sum_{n=1}^\infty a_n e_n$.

Now we prove that, for every $\delta > 0$ the brick $K_{\mathcal B, \mathcal E}$ contains a finite $\delta$-net. So, fix any $\delta > 0$ and choose by \eqref{eq:ccbvn}, $N \in \mathbb N$ so that for any sequence of scalars $(a_n)_{n>N}$, $|a_n| \leq \varepsilon_n$ one has
\begin{equation} \label{2eqkkdjfn}
\Bigl\| \sum_{n > N} a_n e_n \Bigr\| \leq \frac{\delta}{2} \, .
\end{equation}
Let $K' = K_{\mathcal B, \mathcal E'}$ be the brick which corresponds to the same basis and the following half-height $\mathcal E' = (\varepsilon_n')_{n=1}^\infty$, where $\varepsilon_n' = \varepsilon_n$ for $n \leq N$ and $\varepsilon_n' = 0$ for $n > N$. Using the compactness of the closed bounded subset $K'$ of the finite dimensional space $X' = [e_n]_{n=1}^N$, we choose in $K'$ a finite $\delta/2$-net $S = \{x_1, \ldots, x_m\}$. Show that $S$ is a $\delta$-net in $K_{\mathcal B, \mathcal E}$. Indeed, $S \subseteq K' \subseteq K_{\mathcal B, \mathcal E}$, and therefore, $S \subseteq K_{\mathcal B, \mathcal E}$. Let $x \in K_{\mathcal B, \mathcal E}$ be any element. Choose $k \in \{1, \ldots, m\}$ so that $\|x' - x_k\| < \delta/2$, where $x' = \sum_{n=1}^N e_n^*(x) \, e_n \in X'$. Then, in view of \eqref{2eqkkdjfn}, we obtain
$$
\|x_k - x\| = \Bigl\| x_k - \sum_{n=1}^\infty e_n^*(x) \, e_n \Bigr\| \leq \|x_k - x'\| + \Bigl\| \sum_{n > N} e_n^*(x) \, e_n \Bigr\| < \frac{\delta}{2} + \frac{\delta}{2} = \delta.
$$

It is left to show the equivalence of (4) to the other conditions. Indeed, the implication (4) $\Rightarrow$ (3) is obvious, and back, (3) together with (1) implies (4).
\end{proof}

As a consequence of theorems \ref{pr:kirprad} and \ref{th:crcomp} we get the following result.

\begin{cor} \label{cor:compfin}
All the radii of any compact brick are finite and coincide.
\end{cor}

The implication (3) $\Rightarrow$ (1) of Theorem~\ref{th:crcomp} gives the following Gelfand theorem \cite[Theorem~1.3.4]{KK}.

\begin{cor}[Gelfand's theorem] \label{cor:Gelfand}
If a series $\sum_{n=1}^\infty x_n$ of elements of a Banach space $X$ unconditionally converges then the set of all sums $\sum_{n=1}^\infty \theta_n x_n$, $\theta_n = \pm 1$ is compact.
\end{cor}

Remark also that by the compactness of the convex hull of a compact set \cite[p.~364]{B}, the implication (3) $\Rightarrow$ (1) one can deduce from the above Gelfand theorem.

\section{Entropy}

The entropy of a set $A$ is going to be the infimum of the radii of bricks containing $A$. Depending on a type of bricks, we get different types of entropy.

\begin{definition}
Let $X$ be a Banach space. The \emph{entropy} and the \emph{unconditional entropy} of a subset $A \subseteq X$ is a number or the symbol $\infty$, defined, respectively, by
\begin{enumerate}
  \item $\displaystyle{\mathbf{E}(A) = \inf_{A \subseteq K_{\mathcal B, \mathcal E}} r^{\rm unc}(K_{\mathcal B, \mathcal E})}$ (here the infimum is taken over all bricks containing $A$);
  \item $\displaystyle{\mathbf{E}_0(A) = \inf \bigl\{ r^{\rm unc}(K_{\mathcal B, \mathcal E}): \,\, A \subseteq K_{\mathcal B, \mathcal E}, \,\, K_{\mathcal B, \mathcal E} - 1\mbox{-unconditional brick} \bigr\}.}$
\end{enumerate}
\end{definition}

In the case where no brick (of the corresponding type) of finite unconditional radius contains $A$, we set the corresponding entropy to be equal $\infty$. In particular, if $X$ has no basis then there is no brick in $X$, and hence all subsets of $X$ has infinite entropy.

\subsection{Common properties}

In the following statements we summarize simple properties of entropy.

\begin{prop} \label{pr:estentrr}
Let $X$ be a Banach space. Then
\begin{enumerate}
  \item if $A \subseteq B \subseteq X$ then $\mathbf{E}(A) \leq \mathbf{E}(B)$ and $\mathbf{E}_0(A) \leq \mathbf{E}_0(B)$;
  \item if $A \subseteq X$ then $\mathbf{E}_0(A) \geq \mathbf{E} (A) \geq \sup\limits_{x \in A} \|x\|$;
  \item if $K_{\mathcal B, \mathcal E}$ is a compact brick then $\mathbf{E}(K_{\mathcal B, \mathcal E}) = \sup\limits_{x \in K_{\mathcal B, \mathcal E}} \|x\|$;
  \item if $K_{\mathcal B, \mathcal E}$ is a $1$-unconditional compact brick then $\mathbf{E}_0(K_{\mathcal B, \mathcal E}) = \sup\limits_{x \in K_{\mathcal B, \mathcal E}} \|x\|$.
\end{enumerate}
\end{prop}

\begin{proof}
(1) is obvious.

(2) The left-hand side inequality is obvious. We prove the right-hand side inequality. Let $\mathbf{E}(A) < \infty$, and let $K_{\mathcal B, \mathcal E}$ be any brick with $A \subseteq K_{\mathcal B, \mathcal E}$ and $r^{\rm unc}(K_{\mathcal B, \mathcal E}) < \infty$. Then
$$
r^{\rm unc}(K_{\mathcal B, \mathcal E}) \stackrel{\mbox{\tiny Theorem~\ref{pr:kirprad} (3)}}{=} \sup\limits_{x \in K_{\mathcal B, \mathcal E}} \|x\| \geq \sup\limits_{x \in A} \|x\|.
$$

Hence,
$$
\mathbf{E}(A) = \inf_{A \subseteq K_{\mathcal B, \mathcal E}} r^{\rm unc}(K_{\mathcal B, \mathcal E}) \geq \sup\limits_{x \in A} \|x\|.
$$

(3) and (4) follow from (2) and the fact that a brick does contain itself.
\end{proof}

As a consequence of Theorem~\ref{th:crcomp} we obtain the next statement.

\begin{prop} \label{pr:enmo}
Let $A$ be a subset of a Banach space $X$. If $\mathbf{E}(A) < \infty$ then $A$ is precompact.
\end{prop}

Given a subset $A$ of a Banach space $X$, by $\overline{{\rm absconv} \, (A)}$ we denote the closure of an absolute convex hull of $A$, which by definition equals the least closed absolute convex set in $X$ containing $A$. The next assertion follows from from Proposition~\ref{pr:closed}.

\begin{prop} \label{pr:en1}
For any subset $A$ of a Banach space $X$ one has $\mathbf{E}\bigl( \overline{{\rm absconv} \, (A)} \bigr) = \mathbf{E}(A)$ and $\mathbf{E}_0\bigl( \overline{{\rm absconv} \, (A)} \bigr) = \mathbf{E}_0(A)$.
\end{prop}

\subsection{Sudakov's characteristic}

Following Sudakov \cite{Sud1} and generalizing his notions introduced for an orthonormal basis of a Hilbert space to a normalized basis of a Banach space, we give some definitions.

Let $X$ be a Banach space with a normalized basis $\mathcal B = (e_n)_{n=1}^\infty$ and the biorthogonal functionals $(e_n^*)_{n=1}^\infty$.

\begin{definition}
The \emph{clearances} of a subset $A \subset X$ relatively to the basis $\mathcal B$ are set to be the sequence $\gamma_{\mathcal B, n}(A) \in [0,+\infty]$ defined by
\begin{equation} \label{eq:gabarity}
\gamma_{\mathcal B, n}(A) = \sup_{x \in A} |e_n^*(x)|, \,\,\, n = 1,2, \ldots.
\end{equation}
\end{definition}

\begin{definition}
The \emph{radius} of a subset $A \subset X$ relatively to the basis $\mathcal B$ is the number or symbol $r_\mathcal B (A) \in [0,+\infty]$ defined by
\begin{equation} \label{eq:radii}
r_\mathcal B (A) = \sup_{\theta_n = \pm 1} \Bigl\| \sum_{n=1}^\infty \theta_n \gamma_{\mathcal B, n}(A) \, e_n \Bigr\|
\end{equation}
(here the norm of a divergent series is set to be $\infty$).
\end{definition}

This latter radius of a set generalizes the unconditional radius of a brick. Indeed, if $K_{\mathcal B, \mathcal E}$ is a brick in a Banach space $X$ constructed by a basis $\mathcal B = (e_n)_{n=1}^\infty$ with biorthogonal functionals $e_k^* \in X_0^*$ and half-height $\mathcal E = (\varepsilon_n)_{n=1}^\infty$ then $\gamma_{\mathcal B, n}(K_{\mathcal B, \mathcal E}) = \varepsilon_n$ for all $n \in \mathbb N$, and hence, $r_\mathcal B (K_{\mathcal B, \mathcal E}) = r^{\rm unc}(K_{\mathcal B, \mathcal E})$. Thus, Theorem~\ref{th:crcomp} and item (3) of Theorem~\ref{pr:kirprad} imply that if $K_{\mathcal B, \mathcal E}$ is compact then
$$
r_\mathcal B (K_{\mathcal B, \mathcal E}) = r^{\rm ext}(K_{\mathcal B, \mathcal E}) = r^{\rm unc}(K_{\mathcal B, \mathcal E}) = \sup\limits_{x \in K_{\mathcal B, \mathcal E}} \|x\|.
$$

In \cite{Sud1} the author used the sum of the series $\sum_{n=1}^\infty \gamma_{\mathcal B, n}^2(A)$ instead of the introduced above radius of $A$, which corresponds to the square of the radius for the case of an orthonormal basis of a Hilbert space. Very likely, that in such cases the square root of the sum is not taken just for aesthetic reasons, however it is much more natural to consider the norm of an element as a characteristic of something than the square of the norm. Another observation is that, for an orthonormal basis of a Hilbert space (more general, for a $1$-unconditional basis of a Banach space in the real case) the norm of the sum that appears in the definition of the radius does not depend on the signs $\theta_n$, and hence one may replace the right-hand side of \eqref{eq:radii} with the expression $\bigl\| \sum_{n=1}^\infty \gamma_{\mathcal B, n}(A) \, e_n \bigr\|$.

Remark that the radius of a set $r_\mathcal B (A)$ does depend on the basis $\mathcal B$. Moreover, in \cite{D} the first named author provided an example of a set in a separable Hilbert space the radius of which relatively to a certain basis is finite, and infinite relatively to another one.

\begin{definition}
The \emph{Sudakov characteristic} of a subset $A$ of a Banach space $X$ with a basis is the number or symbol $s(A) \in [0,\infty]$, defined by
$$
s(A) = \sup_\mathcal B r_\mathcal B (A),
$$
where the supremum is taken over all normalized bases $\mathcal B$ of $X$.
\end{definition}

The following statement shows that the entropy of a set can be defined as the Sudakov characteristic, but replacing $\sup$ with $\inf$.

\begin{prop} \label{pr:en2}
For any subset $A$ of a Banach space $X$ the following equalities hold
\begin{enumerate}
  \item $\displaystyle{\mathbf{E}(A) = \inf_\mathcal B r_\mathcal B (A)}$ (here the infimum is taken over all normalized bases $\mathcal B$ of $X$);
  \item $\displaystyle{\mathbf{E}_0(A) = \inf \bigl\{ r_\mathcal B (A): \,\, \mathcal B \,\, \mbox{is a} \,\, 1\mbox{-unconditional basis of} \,\, X \bigr\}.}$
\end{enumerate}
\end{prop}

\begin{proof}
We prove (1) only; item (2) is proved similarly. We prove (1) under the assumption that the set of bricks containing $A$ is nonempty (otherwise both sides of the equality equal $\infty$). So, let $\mathcal B = (e_n)_{n=1}^\infty$ be any normalized basis of $X$ with the biorthogonal functionals $(e_n^*)_{n=1}^\infty$, $A \subseteq X$ any subset. Set $\Gamma_\mathcal B = \Gamma_\mathcal B (A) = \bigl(\gamma_{\mathcal B,n}(A)\bigr)_{n=1}^\infty$, where $\gamma_{\mathcal B,n}(A)$ are the clearances of $A$ relatively to $\mathcal B$, defined by \eqref{eq:gabarity}. By \eqref{eq:radii} and Definition~\ref{def:radii} (2), $r_\mathcal B (A) = r^{\rm unc}(K_{\mathcal B, \Gamma_\mathcal B}).$ Hence, taking into account that $A \subseteq K_{\mathcal B, \Gamma_\mathcal B}$, we obtain
$$
\mathbf{E}(A) = \inf_{A \subseteq K_{\mathcal B, \mathcal E}} r^{\rm unc}(K_{\mathcal B, \mathcal E}) \leq \inf_\mathcal B r^{\rm unc}(K_{\mathcal B, \Gamma_\mathcal B}) = \inf_\mathcal B r_\mathcal B (A).
$$

In order to prove the other side inequality, we fix any normalized basis $\mathcal B_0$ of $X$, and denote by $\Gamma_{\mathcal B_0}$ the clearances of $A$ relatively to $\mathcal B_0$. Since $A \subseteq K_{\mathcal B_0, \Gamma_{\mathcal B_0}}$, one has
$$
\inf_\mathcal B r_\mathcal B (A) \leq r_{\mathcal B_0} (A) \leq r_{\mathcal B_0} (K_{\mathcal B_0, \Gamma_{\mathcal B_0}}) = r^{\rm unc}(K_{\mathcal B_0, \Gamma_{\mathcal B_0}}).
$$

By arbitrariness of $\mathcal B_0$, we get
$$
\inf_\mathcal B r_\mathcal B (A) \leq \inf_{A \subseteq K_{\mathcal B, \Gamma_\mathcal B}} r^{\rm unc}(K_{\mathcal B, \Gamma_\mathcal B}).
$$
It remains to observe that, if $A \subseteq K_{\mathcal B, \mathcal E}$ then $A \subseteq K_{\mathcal B, \Gamma_\mathcal B}$, and therefore,
$$
\inf_{A \subseteq K_{\mathcal B, \Gamma_\mathcal B}} r^{\rm unc}(K_{\mathcal B, \Gamma_\mathcal B}) = \inf_{A \subseteq K_{\mathcal B, \mathcal E}} r^{\rm unc}(K_{\mathcal B, \mathcal E}) = \mathbf{E}(A).
$$
\end{proof}

\subsection{Is the entropy of every precompact set finite?}

Equivalently, is every precompact set contained in a compact brick? Of course, this question is substantial for subsets of a Banach space with a basis. The answer depends on a space. We start with an example of an infinite dimensional Banach space in which every precompact set has finite entropy.

\begin{prop} \label{pr:everyprecompfiniteentr}
Every precompact subset $A$ of the Banach space $c_0$ has finite entropy. Moreover,
$$
\mathbf{E}(A) = \mathbf{E}_0(A) = \sup_{x \in A} \|x\|.
$$
\end{prop}

\begin{proof}
Let $A$ be a precompact subset of $c_0$. Let $\varepsilon_n = \gamma_n(A) = \sup_{x \in A} |e_n^*(x)|$, $n = 1,2, \ldots$ be the clearances of $A$ with respect to the standard basis of $c_0$. One can easily deduce from Lemma~\ref{le:crcomp} that $\lim_{n \to \infty} \varepsilon_n = 0$. Indeed, given any $x_0 \in A$, one has
$$
|e_n^*(x_0)| = \|e_n^*(x_0) \, e_n\| \leq \Bigl\| \sum_{k>n-1} e_k^*(x_0) \, e_k\Bigr\| \leq \sup_{x \in A} \Bigl\| \sum_{k>n-1} e_k^*(x) \, e_k\Bigr\|.
$$
Then by arbitrariness of $x_0 \in A$,
$$
\gamma_n(A) \leq \Bigl\| \sum_{k>n-1} e_k^*(x) \, e_k\Bigr\| \stackrel{\mbox{\tiny by Lemma~\ref{le:crcomp}}}{\longrightarrow} 0
$$
as $n \to \infty$.

Since $\lim_{n \to \infty} \varepsilon_n = 0$, the series $\sum_{n=1}^\infty \varepsilon_n e_n$ unconditionally converges in $c_0$. By Theorem~\ref{th:crcomp}, the brick $K_{\mathcal B, \mathcal E}$ is compact. By the definition of $\varepsilon_n$'s, $A \subseteq K_{\mathcal B, \mathcal E}$, and thus $A$ has finite entropy. Thus,
$$
\mathbf{E}_0(A) \leq r^{\rm unc}(K_{\mathcal B, \mathcal E}) \stackrel{\mbox{\tiny Theorem~\ref{pr:kirprad}.(3)}}{=} \sup_{x \in A} \|x\|.
$$
By (2) of Proposition~\ref{pr:estentrr}, the proof is completed.
\end{proof}

\subsection{A class of compacts in a Hilbert space with infinite entropy}

Let $H$ be a separable infinite dimensional Hilbert space, $p > 0$. By $\mathcal B$ we denote the Borel $\sigma$-algebra of subsets of $H$. Following \cite{VTC}, a probability measure $\mu: \mathcal B \to [0,1]$ is said to have
\begin{itemize}
  \item a \emph{strong $p$-th moment} if $\displaystyle{\int_H \|u\|^p d \mu(u) < \infty}$;
  \item a \emph{weak $p$-th moment} if $\displaystyle{\int_H |(h,u)|^p d \mu(u) < \infty}$ for all $h \in H$.
\end{itemize}

In another terminology, $\mu$ is said to have a strong (resp., weak) order in the above cases. Obviously, if $\mu$ has a strong $p$-th moment then $\mu$ has a weak $p$-th moment. One can show that the converse is not true. We need the following example.

\begin{example} \label{ex:measureweaknonstrong}
There exists a probability measure $\mu: \mathcal B \to [0,1]$ having a weak $4$-th moment which does not have a strong $2$-d moment.
\end{example}

\begin{proof}
Let $(e_n)_{n=1}^\infty$ be an orthonormal basis of $H$. Consider the measure $\mu: \mathcal B \to [0,1]$ focused at the points $\sqrt{n} e_n$ with weights $\frac{6}{\pi^2} \, \frac{1}{n^2}$, $n = 1,2, \ldots$, that is,
$$
\mu = \frac{6}{\pi^2} \sum_{n=1}^\infty \frac{1}{n^2} \, \delta_{\sqrt{n} e_n},
$$
where $\delta_x (A) = 1$ if $A \ni x$ and $\delta_x (A) = 0$ if $A \not \ni x$. Then for every $h \in H$ one has
$$
\int_H |(h,u)|^4 d \mu(u) = \frac{6}{\pi^2} \sum_{n=1}^\infty \frac{1}{n^2} (h, \sqrt{n} e_n)^4 = \frac{6}{\pi^2} \sum_{n=1}^\infty (h,e_n)^4 < \infty.
$$

On the other hand,
$$
\int_H \|u\|^2 d \mu(u) = \frac{6}{\pi^2} \sum_{n=1}^\infty \frac{1}{n^2} \| \sqrt{n} e_n\|^2 = \frac{6}{\pi^2} \sum_{n=1}^\infty \frac1n = \infty.
$$
\end{proof}

The next consideration is related to \cite{D1}.

\begin{lemma} \label{le:D1.1}
Let a probability measure $\mu: \mathcal B \to [0,1]$ have a weak $p$-th moment with $p \geq 2$. Then the mapping $T: H \to L_p(H, \mathcal B, \mu)$ given by $(Tu)(h) = (u,h)$ for $u,h \in H$ is a linear bounded operator.
\end{lemma}

\begin{proof}
The linearity of $T$ is obvious, and the boundedness one can prove using the Closed Graph Theorem.
\end{proof}

Recall the definition of the Pettis integral. Let $f: \Omega \to X$, where $(\Omega, \Sigma, \nu)$ is a probability space and $X$ a real Banach space. By the Dunford theorem \cite[p.~52]{DU}, if $x^*(f) \in L_1(\nu)$ for all $x^* \in X^*$ then for each $A \in \Sigma$ there exists $z \in X^{**}$ such that
$$
z(x^*) = \int_A x^*(f) \, d \nu
$$
for all $x^* \in X^*$. The above element $z$ is called the Dunford integral of $f$ over $A$ and is denoted by $z =$ {\tiny D}-\hspace{-0,1 cm}$\displaystyle{\int_A f \, d \nu}$. If, furthermore, $z$ belongs to the canonical image of $X$ in $X^{**}$ then $f$ is said to be Pettis integrable, and the Dunford integral of $f$ is set to be the Pettis integral {\tiny P}-\hspace{-0,1 cm}$\displaystyle{\int_A f \, d \nu}$. The latter condition always holds if $X$ is reflexive, so the Pettis integral exists for any function $f: \Omega \to X$ such that $x^*(f) \in L_1(\nu)$ for all $x^* \in X^*$.

\begin{lemma} \label{le:D1.2}
Let a probability measure $\mu: \mathcal B \to [0,1]$ have a weak $p$-th moment with $p \geq 2$. Then for every $u \in H$ the Pettis integral

\begin{equation} \label{eq:star}
j(u) = \mbox{{\tiny P}-\hspace{-0,2 cm}}\int_H (u,v) \, v \, d \mu(v)
\end{equation}
exists, and $j: H \to H$ is a linear bounded operator. If, moreover, $p > 2$ then $j$ is compact.
\end{lemma}

\begin{proof}
The existence of the Pettis integral in Lemma~\ref{le:D1.2} means that $\bigl(h, (u, \cdot) \, \cdot \bigr) \in L_1(H, \mathcal B, \mu)$ for all $h \in H$, which follows from the well known inequalities
\begin{align*}
\int_H \bigl| (u,v) (h,v) \bigr| \, d \mu(v) &\leq \Bigl( \int_H \bigl| (u,v) \bigr|^2 \, d \mu(v) \Bigr)^{1/2} \Bigl( \int_H \bigl| (h,v) \bigr|^2 \, d \mu(v) \Bigr)^{1/2}\\
&\leq \Bigl( \int_H \bigl| (u,v) \bigr|^p \, d \mu(v) \Bigr)^{1/p} \Bigl( \int_H \bigl| (h,v) \bigr|^p \, d \mu(v) \Bigr)^{1/p} < \infty,
\end{align*}
because $\mu$ has a weak $p$-th moment and $p > 2$. The compactness of $j$ is proved in \cite{D1}.
\end{proof}

Remark that the image $j(H)$ for $\mu$ serves as the space of acceptable shifts for a Gaussian measure $\mu$ (see \cite{D1}).

The following example shows that $j$ need not be a Hilbert-Schmidt operator (an operator $T \in \mathcal L(H)$ is called a Hilbert-Schmidt operator if $\sum_{n=1}^\infty \|Te_n\|^2 < \infty$ for some, or equivalently each, orthonormal basis $(e_n)_{n=1}^\infty$ of $H$).

\begin{example} \label{ex:nonHS}
There exists a probability measure $\mu: \mathcal B \to [0,1]$ having a weak $2$-th moment such that the operator $j$ defined by \eqref{eq:star} is compact and not a Hilbert-Schmidt operator.
\end{example}

\begin{proof}
Fix any orthonormal basis $(e_n)_{n=1}^\infty$ of $H$ and set
$$
\mu = C \sum_{n=1}^\infty \frac{1}{n \ln^2 (n+1)} \, \delta_{\sqrt{n} e_n},
$$
where $C$ is the norming constant taken from the condition $\mu(H) = 1$. Show that $\mu$ has a weak $2$-th moment. Indeed, given any $h \in H$, we have
$$
\int_H |(h,u)|^2 d \mu(u) = C \sum_{n=1}^\infty \frac{1}{n \ln^2 (n+1)} (h, \sqrt{n} e_n)^2 \leq \frac{C}{\ln^2 2} \sum_{n=1}^\infty (h,e_n)^2 < \infty.
$$

Observe that, by the definition of the Pettis integral, for every $u, h \in H$
$$
\bigl( j(u),h \bigr) = \int_H \bigl( h, (u,v) \, v \bigr) \, d \mu(v) = \int_H (u,v) (h,v) \, d \mu(v).
$$

Hence, for every $u \in H$
\begin{align*}
j(u) &= \sum_{n=1}^\infty \bigl( j(u),e_n \bigr) \, e_n\\
&= \sum_{n=1}^\infty \Bigl( \int_H (u,v) (e_n,v) \, d \mu(v) \Bigr) \, e_n\\
&= C \sum_{n=1}^\infty \Bigl( \sum_{k=1}^\infty \frac{(u, \sqrt{k} e_k) (e_n, \sqrt{k} e_k)}{k \ln^2 (k+1)} \Bigr) \, e_n\\
&= C \sum_{n=1}^\infty \frac{n (u,e_n)}{n \ln^2 (n+1)} \, e_n = C \sum_{n=1}^\infty \frac{(u,e_n)}{\ln^2 (n+1)} \, e_n.
\end{align*}

Thus, $j(e_k) = \frac{e_k}{\ln^2(k+1)}$, confirming that $j$ is compact and fails to be a Hilbert-Schmidt operator.
\end{proof}

\bibliographystyle{amsplain}

\end{document}